\documentclass[12pt]{amsart}

\usepackage[margin=1.5in]{geometry}

\usepackage{amsmath}
\usepackage{amsfonts}
\usepackage{amssymb}
\usepackage[all,cmtip]{xy}
\usepackage{longtable}
\usepackage{amsthm}
\usepackage{enumerate}
\usepackage{graphicx}
\usepackage{mathrsfs}
\usepackage{mathrsfs}
\usepackage{xcolor}

\DeclareMathOperator{\sm}{sm}

\theoremstyle{plain}
\newtheorem{theorem}{Theorem}[section]
\newtheorem{corollary}[theorem]{Corollary}

\newtheorem{proposition}[theorem]{Proposition}
\theoremstyle{definition}
\newtheorem{definition}[theorem]{Definition}

\newtheorem{rem}[theorem]{Remark}

\makeatletter
\def\ps@pprintTitle{%
	\let\@oddhead\@empty
	\let\@evenhead\@empty
	\let\@oddfoot\@empty
	\let\@evenfoot\@oddfoot
}
\makeatother

\title{A note on multisecants of the Kummer variety of a Jacobian}

\author{Robert Auffarth and Sebastian Rahausen}
\address{R. Auffarth \\Departamento de Matem\'aticas, Facultad de
	Ciencias, Universidad de Chile, Santiago\\Chile}
\email{rfauffar@uchile.cl}
\address{S. Rahausen \\ Departamento de Matemática \\ Universidad Técnica Federico Santa María \\ Santiago \\ Chile}
\email{srahausen@gmail.com}


\begin{document}
	
	\begin{abstract}
		We show that if $C$ is a smooth projective curve and $\mathfrak{d}$ is a $\mathfrak{g}^{n}_{2n}$ on $C$, then we obtain a rational map $\mathrm{Sym}^{n}(C)\dashrightarrow\mathfrak{d}$ whose fibers can be related in an interesting way to Gunning multisecants of the Kummer variety of $JC$. This generalizes previous work done by the first author with Codogni and Salvati Manni. 
	\end{abstract}
	
	\maketitle

	\section{Introduction}
	
	It is a classic fact that the Kummer variety of the Jacobian of a smooth projective curve, canonically embedded into projective space by twice the theta divisor, has a 4-dimensional family of trisecant lines. The existence of these trisecants characterizes Jacobians among indecomposable principally polarized abelian varieties, thus giving a solution to the Schottky Problem (see \cite{K} and \cite{W}, for example). It is also true that for each $\ell\geq 3$ there exists a $(2\ell-2)$-dimensional family of $(\ell-2)$-dimensional linear subvarieties that intersect the Kummer variety in (at least) $\ell$ points, although for $\ell\geq 4$ they do not characterize Jacobians.
	
	In \cite{ACSM}, it was shown that if $(JC,\Theta)$ is the Jacobian of a smooth projective curve $C$ and 
	\[\mathcal{G}\colon\Theta\dashrightarrow\mathbb{P}T_0(JC)^\vee\] 
	\[ \ \ \ p\mapsto dt_{-p}T_p(\Theta)\]
	is the Gauss map of $\Theta$ (which is defined on the smooth locus $\Theta^{\sm}$ of $\Theta$), then if $x,y,z\in\Theta$ are such that their images in the Kummer variety are linearly dependent (and therefore lie on a trisecant line), then $\mathcal{G}$, when defined, is constant on $\{x,y,z\}$. It was also shown that there exist multisecant points of $\Theta$ that lie on the same fiber of the Gauss map, and these could be found by looking at points on the fibers of the Gauss map of high multiplicity (see \cite[Theorem 5.2]{ACSM}).
	
	The purpose of this article is to generalize what was done in \cite{ACSM} to a somewhat more general setting on a Jacobian. First of all, it is well-known (see \cite[p. 342]{GH}) that the Gauss map of $\Theta$ can be identified with the map that takes an effective divisor $D$ of degree $g-1$ and sends it $\mathrm{span}\{\phi_K(D)\}\in(\mathbb{P}^{g-1})^\vee$, where $\phi_K:C\to\mathbb{P}^{g-1}$ is the canonical map of $C$. We can then generalize this by, instead of taking $\phi_K$ to be the canonical morphism, taking it to be the morphism associated to any $\mathfrak{g}^n_d$.
	
	Indeed, if $D$ is a divisor of degree $d$ and $\mathfrak{d}\subseteq|D|$ is an $n$-dimensional linear subsystem, let $\phi_\mathfrak{d}\colon C\to \mathfrak{d}^\vee\simeq\mathbb{P}^{n}$ denote the associated map, let $C^{(n)}:=\mathrm{Sym}^n(C)$ denote the symmetric product of $C$, and define
	\[\xi_\mathfrak{d}\colon C^{(n)}\dashrightarrow\mathfrak{d}\simeq(\mathbb{P}^n)^\vee\]
	\[F\mapsto\mathrm{span}\{\phi_\mathfrak{d}(F)\}.\]
	In this article we will show that many of the results proved in \cite{ACSM} are true precisely because the Gauss map of the theta divisor of a smooth projective curve can be identified with $\xi_\mathfrak{d}$ for $\mathfrak{d}=|K_C|$, where $K_C$ is the canonical bundle. Indeed, if $\eta\in\mathrm{Pic}^n(C)$ and $\mathfrak{d}\subseteq|\eta^{\otimes2}|$ is an $n$-dimensional linear system, then just as in \cite[Section 3]{ACSM} there is a stratification $\mathfrak{d}_{n}\subseteq\mathfrak{d}_{n-1}\subseteq\cdots\subseteq\mathfrak{d}_1\subseteq\mathfrak{d}_0=\mathfrak{d}$ of $\mathfrak{d}$, where $\mathfrak{d}_1$ is the branch locus of $\xi_\mathfrak{d}$, such that, modulo a couple of technicalities, any Gunning multisecants (defined in Theorem \ref{Gunning}) lying on a fiber of $\xi_\mathfrak{d}$ must lie over one of the $\mathfrak{d}_t$ in an explicit way (see Theorem \ref{fiberstrat}). This result generalizes \cite[Proposition 5.3]{ACSM}. Reciprocally, if we look at a fiber over a general element of $\mathfrak{d}_t$, then by looking at points on a fiber of $\xi_\mathfrak{d}$ of a certain multiplicity, we can find certain multisecants (see Theorem \ref{reciprocal}). This generalizes \cite[Theorem 5.2]{ACSM}.
	
	As observed in Theorem \ref{fiberstrat}, we note that the existence of a $\mathfrak{g}^n_{2n}$ on $C$ is somewhat special, since it automatically implies that either $n\geq g$, or $n=g-1$ and $\eta$ is a theta characteristic (which is precisely the situation studied in \cite{ACSM}), or $C$ is hyperelliptic and the $\mathfrak{g}^n_{2n}$ is a multiple of the unique $\mathfrak{g}^1_2$ on $C$.
	
	We hope that this article will shed some light on the role multisecants of the Kummer variety play in the geometry of the cycles $W_n$ on Jacobians.\\
	
	\noindent\textit{Acknowledgements:} We would like to thank Giulio Codogni and Riccardo Salvati Manni for fruitful conversations, as well as an anonymous referee for helpful suggestions that attributed to the readability of the article. The first author was partially supported by ANID-Fondecyt Grant 1220997.
	
	\section{Effective multisecants}
	
	Recall that the \textit{Kummer variety} of an abelian variety $A$ is defined as $K(A):=A/\langle z\mapsto-z\rangle$, the quotient of $A$ by its natural involution. If $\Theta$ is an irreducible principal polarization on $A$, then we get the morphism
	\[\mathrm{Km}\colon A\to|2\Theta|^\vee\simeq\mathbb{P}^{2^g-1}\]
	associated to $|2\Theta|$ whose image is isomorphic to $K(A)$ (see \cite[Theorem 4.8.1]{BL}). Therefore, given an indecomposable principal polarization, we can canonically embed $K(A)$ in $\mathbb{P}^{2^g-1}$.
	
	We now recall Gunning's multisecant formula \cite[Theorem 2]{RG}:
	
	\begin{theorem}\label{Gunning}
		Let $p_1,\ldots,p_\ell,q_1,\ldots,q_{\ell-2}\in C$ be different points, and for $i\leq \ell$, consider $a_i\in JC$ such that
		\[a_i^{\otimes 2}\simeq\mathcal{O}_{C}\left(2p_i+\sum_{j=1}^{\ell-2}q_j-\sum_{k=1}^\ell p_k\right)\]
		\[a_i\otimes a_j\simeq\mathcal{O}_C\left(p_i+p_j+\sum_{j=1}^{\ell-2}q_j-\sum_{k=1}^\ell p_k\right).\]
		Then $\mathrm{Km}(a_1),\ldots,\mathrm{Km}(a_\ell)$ lie on an $(\ell-2)$-plane.
	\end{theorem}
	
	\begin{definition} Any points $a_1,\ldots,a_\ell\in JC$ that satisfy the conditions of Theorem \ref{Gunning} for certain $p_i,q_j\in C$ will be called \textit{Gunning multisecant points}. 
	\end{definition}
	
	This result implies that the Kummer variety of $JC$ has a $(2\ell-2)$-dimensional family of $\ell$-secant $(\ell-2)$-planes. Define
	\[W_n:=\{L\in\mathrm{Pic}^n(C):h^0(L)>0\};\]
	it is well-known that for $n\leq g$, $W_n$ is birational to $C^{(n)}$ via the map $C^{(n)}\to W_n$, $D\mapsto\mathcal{O}_C(D)$. Given any $\eta\in\mathrm{Pic}^n(C)$ we have an isomorphism
	\[\rho_\eta\colon\mathrm{Pic}^n(C)\to\mathrm{Pic}^0(C)=JC\]
	\[\ \ \ L\mapsto L\otimes\eta^{-1}\]
	which sends $W_n$ to a subvariety $W_n^\eta:=\rho_\eta(W_n)$.
	
	\begin{proposition}
		Let $a_1,\ldots,a_\ell\in JC$ be Gunning multisecants for $\ell\geq 3$, and let $n\geq\ell-1$. Then $\bigcap_{i=1}^\ell(W_n\otimes a_i^{-1})$ contains a subvariety that is isomorphic to $W_{n-\ell+1}$. In particular,
		\[\dim \bigcap_{i=1}^\ell(W_n\otimes a_i^{-1})\geq n-\ell+1.\]
	\end{proposition}
	
	\begin{proof}
		
		Let $p_1,\ldots,p_{\ell},q_1,\ldots,q_{\ell-2}\in C$ be different points with $3\leq \ell\leq n+1$, and let $a_1,\ldots,a_\ell\in JC$ be Gunning multisecant points as in Theorem \ref{Gunning}. Let $E$ be an effective divisor of degree $n-\ell+1$, and define 
		\[\eta_i:=a_i^{-1}\otimes\mathcal{O}_C\left(p_i+\sum_{j=1}^{\ell-2}q_j+E\right)\in\mathrm{Pic}^n(C).\]
		
		We first observe the following: the definition of $\eta_i$ does not depend on the $i$ chosen; in other words, $\eta_i\simeq\eta_j$ for all $i$ and $j$. Indeed, we want to show that $a_j\otimes a_i^{-1}\simeq \mathcal{O}_C(p_j-p_i)$. This is true since
		\[a_j\otimes a_i^{-1}\simeq a_j^{\otimes2}\otimes (a_i\otimes a_j)^{-1}\simeq\mathcal{O}_C(p_j)\otimes\mathcal{O}_C(p_j)^{-1}\]
		by the definition of the Gunning multisecant points. We will therefore write $\eta$ instead of $\eta_i$. Notice that for all $i\leq\ell$,
		\[a_i\in W_n^\eta\]
		and the images of the $a_i$ in $K(JC)$ are linearly dependent. By varying $E$ in $\mathrm{Sym}^{n-\ell+1}(C)$, we see that 
		\[\left\{a_i^{-1}\otimes\mathcal{O}_C\left(p_i+\sum_{j=1}^{\ell-2}q_j+E\right):\deg(E)=n-\ell+1,E\geq0\right\}\simeq W_{n-\ell+1},\]
		and the proposition is proved.
	\end{proof}
	
	\begin{corollary}
		Given $n\geq\ell-1$ and Gunning multisecant points $a_1,\ldots,a_\ell\in JC$, there exists an $(n-\ell+1)$-dimensional family $Z_{n}(a_1,\ldots,a_\ell)\subseteq \mathrm{Pic}^{n}(C)$ such that for all $\eta\in Z_{n}(a_1,\ldots,a_\ell)$, $a_1,\ldots,a_\ell\in W_n^\eta$.
	\end{corollary}

	\section{Rational maps defined on $C^{(n)}$}
	
	Let $D$ be an effective divisor of degree $d$, let $\mathfrak{d}\subseteq|D|$ be an $n$-dimensional linear system, and let $\phi_\mathfrak{d}:C\to\mathfrak{d}^\vee\simeq\mathbb{P}^n$ denote the morphism associated to $\mathfrak{d}$. Let $C^{(n)}$ denote the symmetric product $\mathrm{Sym}^n(C)$. As in the introduction, we can define the rational map
	\[\xi_\mathfrak{d}:C^{(n)}\dashrightarrow\mathfrak{d}\simeq(\mathbb{P}^n)^\vee\]
	\[F\mapsto\mathrm{span}\{\phi_\mathfrak{d}(F)\}.\]
	Note that by the General Position Theorem (see \cite[Page 109]{ACGH}), this map is well-defined on a non-empty open subset of $C^{(n)}$. Actually, we see that $\xi_\mathfrak{d}(F)$ is defined if and only if $\dim [\mathfrak{d}\cap(|D-F|+F)]=0$, since this is the same as saying that there is a unique $H\in\mathfrak{d}$ such that $F\leq H$. Since $W_n$ is birational to $C^{(n)}$, we can also define $\xi_\mathfrak{d}$ on $W_n$, but it will be more comfortable and natural in this article to work in general with divisors than line bundles. 
	
	\begin{rem} We note that if $\mathfrak{d}=|K_C|$ is the complete linear system of the canonical bundle of $C$, $\rho:C^{(g-1)}\to\mathrm{Pic}^{g-1}(C)$ is the natural map $E\mapsto\mathcal{O}_C(E)$, $\kappa$ is a theta characteristic and $\mathcal{G}:\Theta=W_{g-1}\otimes\kappa^{-1}\dashrightarrow\mathbb{P}^{g-1}$ is the Gauss map of $\Theta$, then the following diagram commutes (when defined):
		
		\centerline{\xymatrix{
				C^{(g-1)}\ar[dr]^\rho\ar@{-->}[rr]^{\xi_\mathfrak{d}}&&\mathbb{P}^{g-1}\\&W_{g-1}\ar[r]^{\otimes\kappa^{-1}}&\Theta\ar@{-->}[u]_{\mathcal{G}}
		}}
		
	\end{rem}
	\vspace{0.5cm}
	
	In what follows we will collect a few of the main properties of $\xi_\mathfrak{d}$:

	\begin{proposition}\label{map} Let $\deg\phi_\mathfrak{d}$ denote the degree of the extension $k(\phi_\mathfrak{d}(C))\hookrightarrow k(C)$. Then if $n\leq \frac{d}{\deg\phi_\mathfrak{d}}$, we have the following:
		
		\begin{enumerate}
			\item The rational map $\xi_\mathfrak{d}$  is dominant and finite, and the generic fiber is of degree $\binom{d/\deg\phi_\mathfrak{d}}{n}(\deg\phi_\mathfrak{d})^n$.
			\item Let $H=n_1p_1+\cdots+n_rp_r\in\mathfrak{d}$ where all the $p_i$ are different and $n_i\in\mathbb{N}$, and let $F=m_1p_1+\cdots+m_rp_r\in\xi_\mathfrak{d}^{-1}(H)$. Then
			\[\mathrm{mult}_{F}\xi_\mathfrak{d}=\prod_{i=1}^r\binom{n_i}{m_i}.\]
			\item If $\mathfrak{d}$ is very ample, we can recover the embedding $\phi_\mathfrak{d}$ from $\xi_\mathfrak{d}$ in the sense that $C$ is the dual of the branch locus of $\xi_\mathfrak{d}$ and $\phi_\mathfrak{d}$ is just the natural embedding of this locus into projective space. 
			
		\end{enumerate}
	\end{proposition}
	\begin{proof}
		For part (1), to prove the dominance of $\xi_\mathfrak{d}$, let $H$ be a general hyperplane in $\mathbb{P}^n$ such that $\phi_\mathfrak{d}^*H\in Q$ is reduced. We can take $F\leq\phi_\mathfrak{d}^*H$ such that no two points of its support lie on the same fiber of $\phi_\mathfrak{d}$. By the General Position Theorem, $\phi_\mathfrak{d}(F)$ spans $H$.
		
		As for the degree, we see that for a generic hyperplane $H$, by the General Position Theorem we can choose any $n$ of the $d/\deg\phi_\mathfrak{d}$ points of $H\cap\phi_\mathfrak{d}(C)$ in order to generate $H$. Given theses points, we need to choose in each preimage of $\phi_\mathfrak{d}$ exactly one of $\deg\phi_\mathfrak{d}$ points, and this gives us exactly $\binom{d/\deg\phi_\mathfrak{d}}{n}(\deg\phi_\mathfrak{d})^n$ choices.
		
		As for part (2), this follows from a similar argument to the one presented in the proof of \cite[Prop. 4.7]{ACSM}.
		
		For the last statement, we have that since $\mathfrak{d}$ is very ample, the curve $C_\mathfrak{d}:=\phi_\mathfrak{d}(C)$ is isomorphic to $C$ and therefore is smooth. Note that the branch locus $B$ of $\xi_\mathfrak{d}$ is equal to the dual hypersurface $C_\mathfrak{d}^*$ in $(\mathbb{P}^n)^\vee$, i.e., the set of hyperplanes in $\mathbb{P}^n$ tangent to $C_\mathfrak{d}$ (see \cite[p. 821]{A}). Using the Reflexivity Theorem \cite[Theorem 1.7]{T}, $C$ embedded in $\mathbb{P}^n$ (and therefore $\phi_\mathfrak{d}$) can be reconstructed through the branch locus of $\xi_\mathfrak{d}$, namely, $C\simeq C_\mathfrak{d}=C_\mathfrak{d}^{**}=B^*$.
	\end{proof}

	For $0\leq t\leq \lfloor\frac{d}{2}\rfloor$, following \cite[Section 3]{ACSM} (although with different indexes), let us consider
	\[\mathfrak{d}_t:=\left\{F\in\mathfrak{d}:F\text{ is of the form }F=\sum_{i=1}^{d-2t}P_i+2\sum_{j=1}^{t}Q_j\right\}.\]
	This gives us a stratification of $\mathfrak{d}$:
	\[\mathfrak{d}_{\lfloor\frac{d}{2}\rfloor}\subseteq\cdots\subseteq\mathfrak{d}_2\subseteq\mathfrak{d}_1\subseteq\mathfrak{d}_0=\mathfrak{d}.\]
	We note moreover that $\mathfrak{d}_1$ is exactly the branch locus of $\xi_\mathfrak{d}$. We can now state our first main theorem:
	
	\begin{theorem}\label{fiberstrat}
		Let $n\geq\ell-1$, $\eta\in\mathrm{Pic}^n(C)$, let $a_1,\ldots,a_\ell\in JC$ be Gunning multisecant points such that $\eta\in Z_n(a_1,\ldots,a_\ell)$ and let $\mathfrak{d}\subseteq|\eta^{\otimes2}|$ be an $n$-dimensional linear system. Let $x_i\in C^{(n)}$ be such that $\mathcal{O}_C(x_i)=a_i\otimes\eta$. Then if $\xi_\mathfrak{d}$ is defined and constant on $\{x_1,\ldots,x_\ell\}$, we have that $\xi_\mathfrak{d}(x_i)\in\mathfrak{d}_{n-\ell+1}$.
	\end{theorem}
	
	\begin{proof} We can take the $a_i$ such that
		\[a_i^{\otimes 2}\simeq\mathcal{O}_{C}\left(2p_i+\sum_{j=1}^{\ell-2}q_j-\sum_{k=1}^\ell p_k\right)\]
		\[a_i\otimes a_j\simeq\mathcal{O}_C\left(p_i+p_j+\sum_{j=1}^{\ell-2}q_j-\sum_{k=1}^\ell p_k\right)\]
		for certain points $p_i$ and $q_j$ on $C$. We can also write
		\[x_i=p_i+\sum_{j=1}^{\ell-2}q_j+E\]
		for some $E\in C^{(n-\ell+1)}$ as in the previous section. Since $\xi_\mathfrak{d}$ is defined and constant on $x_1,\ldots,x_\ell$, we have that 
		\[H:=\xi_\mathfrak{d}(x_i)\in\mathfrak{d}\]
		is independent of $i$, and $x_i\leq H$ for all $i$. In particular, $p_i\leq H$ for all $i$. We can therefore write
		\[H=\sum_{i=1}^\ell p_i+\sum_{j=1}^{\ell-2}q_j+T\]
		for some effective divisor $T$. Now since for each $k$, $H$ is the unique element of $\mathfrak{d}$ that contains $x_k$, (since $\xi_\mathfrak{d}$ is defined at $x_k$), we have that there exists a unique effective divisor $y_k$ such that $H=x_k+y_k$. We see that
		\[\sum_{i=1}^{\ell}p_i-p_k+T=E+y_k\]
		for all $k$, and therefore we can conclude that $E\leq T$. Therefore, we have that
		\[H=\sum_{i=1}^{\ell}p_i+\sum_{j=1}^{\ell-2}q_j+E+E'\]
		for some effective divisor $E'$. On the other hand, since by hypothesis $\eta$ is chosen such that $\eta\simeq\mathcal{O}_C(x_i)\otimes a_i^{-1}$, we have that 
		\[\eta^{\otimes2}\simeq\mathcal{O}_C\left(\sum_{i=1}^{\ell}p_i+\sum_{j=1}^{\ell-2}q_j+2E\right).\]
		Since $y_1=\sum_{i=2}^\ell p_i+E'$ is unique, we can conclude that $E'=E$. In particular, $H\in\mathfrak{d}_{n-\ell+1}$.
		
	\end{proof}
	\begin{rem}\label{cases}
		Note that under the hypotheses of Theorem \ref{fiberstrat}, by Clifford's Theorem and Riemann-Roch, $\eta$ must satisfy exactly one of the following:
		\begin{enumerate}
			\item $n=g-1$ and $\eta^{\otimes2}=K_C$ is the canonical bundle.
			\item $h^0(\eta^{\otimes 2})=2n+1-g$ (i.e. $\eta^{\otimes2}$ is non-special) and $n\geq g$.
			\item $C$ is hyperelliptic and $\eta^{\otimes2}$ is a multiple of the unique $\mathfrak{g}^1_2$ on $C$.
		\end{enumerate}
		Indeed, either $\eta^{\otimes2}$ is non-special and so we are in case (2), or either $n=g-1$ and it is the canonical bundle, or $C$ is hyperelliptic and the $\mathfrak{g}^n_{2n}$ is a multiple of the unique $\mathfrak{g}^1_2$ on $C$.
	\end{rem}
	
	\begin{rem}
		Note that in \cite[Proposition 5.3]{ACSM}, the same result of Theorem \ref{fiberstrat} was proved in Case (1) of Remark \ref{cases}. Moreover, in this case it was shown that for $\ell=3$, it is indeed true that trisecant points lie on the same fiber of the Gauss map, when defined. 
	\end{rem}
	
	We can also get a reciprocal in some sense:
	
	\begin{theorem}\label{reciprocal}
		Let $d=2n$ and $H=\sum_{i=1}^{2\ell-2}P_i+2E\in\mathfrak{d}_{n-\ell+1}$, where $E=n_1z_1+\cdots+n_rz_r$ and $n_i\in\mathbb{N}$ and such that the $P_i$ and $z_j$ are different. Take $\eta\in\mathrm{Pic}^n(C)$ such that $\mathfrak{d}\subseteq|\eta^{\otimes2}|$. Then given a partition $\Sigma$ of $\{P_1,\ldots,P_{2\ell-2}\}$ into a set of cardinality $\ell$ and another of cardinality $\ell-2$, we obtain points $x_1^\Sigma,\ldots,x_\ell^\Sigma\in C^{(n)}$ such that 
		\begin{enumerate}
			\item $\xi_\mathfrak{d}$ is constant on $\{x_1^\Sigma,\ldots,x_\ell^\Sigma\}$, when defined, with image $H$.
			\item $\mathrm{mult}_{x_i^\Sigma}\xi_\mathfrak{d}=\prod_{i=1}^r\binom{2n_i}{n_i}$ if $\xi_\mathfrak{d}$ is defined at $x_i^\Sigma$.
			\item If $a_i:=\mathcal{O}_C(x_i^\Sigma)\otimes\eta^{-1}\in W_n^\eta$, then the images of $a_1,\ldots,a_\ell$ in the Kummer variety are linearly dependent.
		\end{enumerate}
		
		On the other hand, if $F\in\xi_\mathfrak{d}^{-1}(H)$ is such that $\mathrm{mult}_{F}\xi_\mathfrak{d}=\prod_{i=1}^r\binom{2n_i}{n_i}$, then $F=x_j^\Sigma$ for some partition $\Sigma$ and some $j$. 
	\end{theorem}
	
	\begin{proof} Let 
		\[\Sigma=\{p_1,\ldots,p_\ell\}\sqcup\{q_1,\ldots,q_{\ell-2}\}\]
		be a partition of $\{P_1,\ldots,P_{2\ell-2}\}$, and define
		\[x_i^\Sigma:=p_i+\sum_{j=1}^{\ell-2}q_j+E\leq H.\]
		We note then that if $\xi_\mathfrak{d}$ is defined on $x_i^\Sigma$, then
		\[\xi_\mathfrak{d}(x_i^\Sigma)=H.\]
		
		Moreover, by Proposition \ref{map}, we have that the multiplicity of $\xi_\mathfrak{d}$ at $x_i^\Sigma$ is exactly $\prod_{i=1}^r\binom{2n_i}{n_i}$. This proves items (1) and (2).
		
		For part (3), let $a_1,\dots,a_l\in W_n^\eta$ be as in the theorem. A simple computation shows that
		\[a_i^{\otimes 2}\simeq\mathcal{O}_{C}\left(2p_i+\sum_{j=1}^{\ell-2}q_j-\sum_{k=1}^\ell p_k\right)\]
		\[a_i\otimes a_j\simeq\mathcal{O}_C\left(p_i+p_j+\sum_{j=1}^{\ell-2}q_j-\sum_{k=1}^\ell p_k\right).\]
		Then we have that $a_1,\ldots,a_\ell$ are Gunning multisecant points.
		
		For the last statement, we see that, if $\mathrm{mult}_{F}\xi_\mathfrak{d}=\prod_{i=1}^r\binom{2n_i}{n_i}$, we can retrace our steps and obtain that $F=x_j^\Sigma$ for some partition $\Sigma$ and some $j$.
	\end{proof}


\begin{thebibliography}{ABCD000}
		
		\bibitem[A58]{A} A. Andreotti, \textit{On a theorem of Torelli}, Amer. J. of Math. 80 (1958) 801--828.
		
		\bibitem[ACGH85]{ACGH} E. Arbarello, M. Cornalba, P. Griffiths and J. Harris, \textit{Geometry of Algebraic Curves}, I, Grundl. der Math. Wiss.,
		267, Springer-Verlag, New York, 1985.
		
		\bibitem[ACSM19]{ACSM} R.~Auffarth, G.~Codogni, R.~Salvati~Manni. \textit{The Gauss map and secants of the Kummer variety.} Bull. London Math. Soc. 51 (2019), 489--500. 
		
		\bibitem[BL04]{BL} Ch.~Birkenhake, H.~Lange. \textit{Complex abelian varieties.}
		Second edition. Grundlehren der Mathematischen Wissenschaften 302. \textit{Springer-Verlag, Berlin}, 2004.
		
		\bibitem[GH11]{GH} P. Griffiths and J. Harris, \textit{Principles of Algebraic Geometry}, Wiley Classics
		Library, Wiley, 2011.
		
		\bibitem[G86]{RG} R. Gunning. \textit{Some identities for abelian integrals}, Amer. J. Math. (1986) 39-74.
		
		\bibitem[K10]{K}I. Krichever. \textit{Characterizing Jacobians via trisecants of the Kummer variety}, Ann. of Math. 172 (2010), no. 1, 485-516
		
		
		\bibitem[T05]{T} E. Tevelev. \textit{Projective Duality and Homogeneous Spaces}, Encyclopaedia of Mathematical Sciences, Springer Verlag, Berlin, 2005.
		
		\bibitem[W84]{W}G. Welters. \textit{A criterion for Jacobi varieties}. Ann. of Math. (2), 120(3):497-504, 1984.
		
	\end{thebibliography}
\end{document}